\documentclass[12pt]{amsart}
\usepackage[english]{babel}
\usepackage[utf8]{inputenc}
\usepackage[T1]{fontenc}
\usepackage[nomath]{lmodern}

\usepackage{amssymb,amsmath,amsthm,amsfonts}
\usepackage{thmtools,mathtools}
\usepackage{microtype}
\allowdisplaybreaks
\usepackage{enumitem} 

\usepackage[dvipsnames]{xcolor}
\usepackage{graphicx} 
\usepackage{tikz,tikz-cd,pgfplots}
\pgfplotsset{compat=newest}

\usepackage{caption} 
\usepackage{fmtcount} 
\usepackage[breaklinks,pdfencoding=auto,psdextra]{hyperref}
\usepackage{comment} 

\addto\extrasenglish{}
\addto\extrasenglish{}
\addto\extrasenglish{\def\equationautorefname~#1\null{\textnormal{(#1)}\null}}
\declaretheorem[name=Theorem,refname={Theorem},style=plain,numberwithin=section]{theorem}
\declaretheorem[name=Theorem,refname={Theorem},style=plain,numbered=no]{theorem*}
\declaretheorem[name=Proposition,refname={Proposition},style=plain,sibling=theorem]{proposition}
\declaretheorem[name=Proposition,refname={Proposition},style=plain,numbered=no]{proposition*}
\declaretheorem[name=Lemma,refname={Lemma},style=plain,sibling=theorem]{lemma}
\declaretheorem[name=Lemma,refname={Lemma},style=plain,numbered=no]{lemma*}
\declaretheorem[name=Definition,refname={Definition},style=definition,sibling=theorem]{definition}
\declaretheorem[name=Definition,refname={Definition},style=definition,numbered=no]{definition*}
\declaretheorem[name=Remark,refname={Remark},style=definition,sibling=theorem]{remark}
\declaretheorem[name=Remark,refname={Remark},style=remark,numbered=no]{remark*}

\declaretheorem[name=Example,refname={Example},style=definition,numbered=no]{example*}

\declaretheorem[name=Corollary,refname={Corollary},style=plain,numbered=no]{corollary*}
\declaretheorem[name=Preliminaries,refname={Preliminaries},style=definition,numbered=no]{Preliminaries*}
\declaretheorem[name=Acknowledgements,refname={Acknowledgements},style=definition,numbered=no]{Acknowledgements*}

\declaretheorem[name=Conjecture,refname={Conjecture},style=plain,sibling=theorem]{Conjecture}

\def\CC{{\mathbb{C}}}

\def\bP{{\mathbf{P}}}

\def\bZ{{\mathbf{Z}}}

\def\cC{{\mathcal{C}}}
\def\cD{{\mathcal{D}}}

\def\cM{{\mathcal{M}}}
\def\cN{{\mathcal{N}}}
\def\cO{{\mathcal{O}}}
\def\cP{{\mathcal{P}}}
\def\cQ{{\mathcal{Q}}}

\def\cU{{\mathcal{U}}}
\def\cV{{\mathcal{V}}}
\def\cW{{\mathcal{W}}}

\def\AA{{\mathbb{A}}}
\def\ZZ{{\mathbb{Z}}}
\def\QQ{{\mathbb{Q}}}
\def\CC{{\mathbb{C}}}

\def\PGL{\operatorname{PGL}}

\def\Bl{\operatorname{Bl}}
\def\Br{\operatorname{Br}}

\def\codim{\operatorname{codim}}

\def\GL{\operatorname{GL}}
\def\Gr{\operatorname{Gr}}

\def\Hom{\operatorname{Hom}}

\def\rank{\operatorname{rank}}
\def\Sing{\operatorname{Sing}}
\def\SL{\operatorname{SL}}

\def\Spec{\operatorname{Spec}}

\def\Bl{\operatorname{Bl}}
\def\Pf{\operatorname{Pf}}

\def\KKK{{\mathrm{K3}}}
\def\prim{{\mathrm{prim}}}

\def\trans{{\mathrm{trans}}}
\def\van{{\mathrm{van}}}

\let\ordexists\exists
\def\exists{\operatorname{\ordexists}}
\let\ordforall\forall
\def\forall{\operatorname{\ordforall}}

\def\setmid#1#2{{\left\{{#1}\;\middle|\;{#2}\right\}}}

\def\tilde{\widetilde}
\def\setminus{\smallsetminus}

\def\git{/\!\!/}

\def\bw#1{\mathchoice%
 {\textstyle{\bigwedge\mkern-4.5mu^{#1}\mkern1mu}}%
 {\textstyle{\bigwedge\mkern-4.5mu^{#1}\mkern1mu}}%
 {\scriptstyle{\bigwedge\mkern-5mu^{#1}}}%
 {\scriptscriptstyle{\bigwedge\mkern-5mu^{#1}}}%
}

\def\longarrow#1#2{\mathchoice{#2}{#1}{#1}{#1}}
\def\to{\longarrow{\rightarrow}{\longrightarrow}}

\let\shortmapsto\mapsto
\def\mapsto{\longarrow{\shortmapsto}{\longmapsto}}

\usepackage[margin=1.2in]{geometry}

\title{Rationality of Peskine varieties}

\author[V.~Benedetti]{Vladimiro Benedetti}

\author[D.~Faenzi]{Daniele Faenzi}

\address{Institut de Mathématiques de Bourgogne,
UMR CNRS 5584,
Université de Bourgogne et Franche-Comté,
9 Avenue Alain Savary,
BP 47870,
21078 Dijon Cedex,
France}
\email{vladimiro.benedetti@u-bourgogne.fr}
\email{daniele.faenzi@u-bourgogne.fr}

\begin{document}
\begin{abstract}
We study the rationality of the Peskine sixfolds in $\bP^9$. We prove the rationality of the Peskine sixfolds in the divisor $\cD^{3,3,10}$ inside the moduli space of Peskine sixfolds and we provide a cohomological condition which ensures the rationality of the Peskine sixfolds in the divisor $\cD^{1,6,10}$ (notation from \cite{BenSong}). We conjecture, as in the case of cubic fourfolds containing a plane, that the cohomological condition translates into a cohomological and geometric condition involving the Debarre-Voisin hyperkähler fourfold associated to the Peskine sixfold.
\end{abstract}

\maketitle

\section{Introduction}

The relationship between birational geometry and cohomological invariants of projective varieties has been widely investigated. In this respect, in view of their (partially conjectural) connection with rationality of fourfolds and hyperkähler geometry, a particular emphasis has been put on varieties whose derived category has a K3 component, namely a semiorthogonal factor having a Serre functor which is a shift by 2, see for instance \cite{kuz} or \cite{AT}. This is particularly relevant for certain families of Fano varieties, like  cubic fourfolds, Gushel-Mukai varieties etc, related to locally complete families of hyperkähler manifolds. Indeed, arithmetical considerations allow precise expectations about the connection between the rationality of the Fano manifolds and the existence of an actual K3 surface whose derived category is the K3 semiorthogonal factor mentioned above. Here we focus on one of these families, namely the Debarre-Voisin varieties, and their relationship with a Fano sixfold, called Peskine variety, in order to elucidate the behaviour of these varieties along some divisors of the moduli space. We look mostly at the rationality of the Peskine variety and try to relate it with the vanishing of certain Brauer classes, leaving the study of derived category for future work. \newline To formulate the results more precisely, let us first recall that Peskine and Debarre-Voisin varieties are associated with trivectors $\sigma$ in 10 variables. Given a general element $\sigma\in \wedge^3 \CC^{10}$ one can construct, among other varieties, the Peskine sixfold $X_1^\sigma\subset \bP^9$ and the Debarre-Voisin fourfold $X_6^\sigma\subset \Gr(6,10)$ ($\Gr$ denoting the Grassmannian). The former is a Fano variety, while the latter is a hyperkähler variety of $K3$-type. Surprisingly enough, these two varieties share some very interesting features: indeed, from a Hodge theoretic point of view, it was shown in \cite{BenSong} that their middle integral Hodge structure is isomorphic. The Peskine is not the only Fano variety closely related to the Debarre-Voisin hyperkähler variety (for instance, from $\sigma$ one can also construct a hyperplane section $X_3^\sigma$ in the Grassmannian $\Gr(3,10)$, with analogous results concerning its Hodge structure), but it is probably the easiest to study by virtue of its low dimension.

The Peskine-Debarre-Voisin relationship is very similar to the case of cubic fourfolds, which are Fano varieties whose variety of lines is a hyperkähler fourfold of type $K3$ (\cite{bd}), or to the case of Gushel-Mukai varieties (\cite{dk}). In the context of cubic fourfolds, a few conjectures (\cite{hassett}, \cite{Addington}, \cite{kuz}) have tried to relate the locus of cubic fourfolds which are rational to some special divisors in the moduli space of hyperkähler varieties of $K3$-type. This paper is an attempt to extend this rationality point of view to the case of Debarre-Voisin varieties and Peskine sixfolds (which play the role of cubic fourfolds). For this reason we focus on the rationality of Peskine varieties.
\medskip

Let $V_n$ be a $n$-dimensional complex vector space with $n$ even, let $\sigma\in \wedge^3 V_n^\vee$ be general and define a Peskine variety of dimension $n-4$ as \[X_{1,n-4}^\sigma\coloneqq \{[U_1]\in \bP(V_{n}) \mid \rank (\sigma(U_1,V_{n},V_{n}))\leq n-4\},\] where $\CC \cong U_1\subset V_n$. $X_{1,n-4}^\sigma$ can also be defined as the singular locus of the Pfaffian hypersurface in $\bP(V_n)$ associated to $\sigma$. Its canonical bundle is equal to $K_{X_{1,n-4}^\sigma}=\cO(-3)|_{X_{1,n-4}^\sigma}$. In low dimension the rationality of $X_{1,n-4}^\sigma$ is easily dealt with. Indeed, $X_{1,0}^\sigma$ is a point, $X_{1,2}^\sigma$ is the disjoint union of two copies of $\bP^2$ and $X_{1,4}^\sigma$ is isomorphic to $\bP^2\times \bP^2$. Thus the first non trivial Peskine variety (and the first one for which the rationality question is interesting) is the Peskine sixfold, which we will simply denote by $X_1^\sigma$.
\medskip

The moduli space of Peskine sixfolds (and of Debarre-Voisin varieties) is birational to the GIT quotient $$\cM\coloneqq \bP\big(\bw3V_{10}^\vee\big)\git \SL(V_{10}).$$ Three $\PGL_{10}$-stable divisors of this moduli space where constructed in \cite{dv} and then studied in detail in \cite{BenSong}.  In the notation of the latter mentioned paper, let us denote them by $\cD^{3,3,10}$, $\cD^{1,6,10}$, $\cD^{4,7,7}$ (the precise definitions will be recalled later on in the paper). 

We will be interested in the first two divisors, because for both of them it is possible to construct a possibly twisted $K3$ surface associated to $X_6^\sigma$. If $\sigma \in \cD^{3,3,10}$ then $X_6^\sigma$ is singular and birational to the Hilbert scheme of two points $S_{22}^{[2]}$ over a $K3$ surface $S_{22}$ of degree $22$ (\cite{dv}). If $\sigma \in \cD^{1,6,10}$ then $X_6^\sigma$ is isomorphic to a moduli space of twisted sheaves on a \emph{twisted $K3$ surface} $(S_{6},\beta)$ of degree $6$ (\cite{BenSong}). The two main results about rationality of the Peskine sixfold are the following:

\begin{theorem*}[Theorem \ref{thm_rationality_3310}]
If $\sigma\in \cD^{3,3,10}$ is a general element then $X_1^\sigma$ is rational.
\end{theorem*}

\begin{theorem*}[Propositions \ref{prop_birational_1.1}, \ref{prop_birationality_1.2}, \ref{prop_description_1.2_quadrics}, Theorem \ref{thm_rat_X1}]
If $\sigma\in \cD^{1,6,10}$ is a general element then $X_1^\sigma$ is birational to a variety $X_{1.2}^\sigma$ which admits a dominant morphism to $\bP^3$ whose general fiber is the intersection of two quadrics in $\bP^5$. 

If there exists a $4$-dimensional subvariety $Z\subset X_{1.2}^\sigma$ whose intersection with a hyperplane section in the general fiber is odd and $Z\times_{\bP^3}\Spec(\CC(\bP^3)) $ is geometrically integral then $X_1^\sigma$ is rational.
\end{theorem*}

Moreover, if $\sigma\in \cD^{1,6,10}$ Proposition \ref{prop_trivial_twisting} gives a geometric condition on $X_6^\sigma$ in order to have that the twisting $\beta$ is trivial. Speculating as in the case of cubics and the corresponding rationality conjectures, in Section \ref{sec_open_questions} we conjecture that the triviality of the $\beta$-twisting and the rationality of $X_1^\sigma$ for $\sigma \in \cD^{1,6,10}$ are closely related.
\medskip

After reviewing what is known for cubic fourfolds, we study the rationality of the Peskine sixfold when $\sigma \in \cD^{3,3,10}$ (and we prove Theorem \ref{thm_rationality_3310}) in Section \ref{sec:D22} and when $\sigma\in \cD^{1,6,10}$ in Section \ref{sec:D24}. In particular, in the latter we recall the results in \cite{BenSong}, we provide some birational models of $X_1^\sigma$ and then we prove Theorem \ref{thm_rat_X1}. In Section \ref{sec_open_questions} we formulate some conjectures and open questions that arise from this work.

\begin{Acknowledgements*}
The authors wish to thank Jieao Song for useful discussions and for suggesting the general idea in the proof of Theorem \ref{thm_rationality_3310}. This project was partially supported by the EIPHI Graduate School (contract ANR-17-EURE-0002), by ANR-20-CE40-0023, by SupToPhAG/EIPHI ANR-17-EURE-0002, Région Bourgogne-
Franche-Comté, Feder Bourgogne and by BRIDGES ANR-21-CE40-0017.
\end{Acknowledgements*}


\begin{Preliminaries*}
We will denote by $V_i,U_i,W_i$ vector spaces of dimension $i$. If $E$ is a vector bundle on a variety $X$, $\bP(E)$ will denote the projective bundle of lines in $E$. On the Grassmannian $\Gr(k,V)$ we will denote by $\cU$ the tautological bundle of rank $k$, and by $\cQ$ the corresponding quotient bundle. On $\bP(V)$ the tautological bundle will also be denoted by $\cO(-1)$. 
\end{Preliminaries*}

\subsection{Review: Rationality of cubic fourfolds}

Rationality of cubic fourfolds is one open question in algebraic geometry that has raised a lot of interest in the last years. A conjecture by Hassett states that a cubic fourfold should be rational if and only if there exists an associated K3 surface to the cubic (see \cite{hassett}). An equivalent formulation (see \cite{Addington}) is that a cubic is rational if and only if its Fano variety of lines is birational to some moduli space of sheaves on a K3 surface. In particular, the very general cubic fourfold should be irrational, and rational cubics should be contained in the union of countably many divisors in the moduli space of cubics. Beware that if the variety of lines is birational to a moduli space of \emph{twisted} sheaves on a K3 surface (with non-trivial twisting), the cubic fourfold should not be rational. Finally, let us mention that this conjecture can be stated also in terms of the derived category of the cubic fourfold (see \cite{kuz} and \cite{AT}). 

The rationality side of the conjecture has been proven in a few significant examples, for instance for pfaffian cubics, for singular cubics, for cubics containing a plane and another cycle of odd intersection with the plane (see \cite{hassett} for a survey). Let us briefly recall here the case of singular cubics and cubics containing a plane, since they will be inspirational for the results in the present paper.

\subsubsection{Singular cubics}
This case is the easiest one in terms of rationality. On the one side, a general singular cubic has a single nodal point; then projection from the nodal point gives a birational map from the cubic to $\bP^4$. On the other side, to any singular cubic with a single nodal point one can associate a K3 surface of degree six such that the Fano variety of lines is birational to the Hilbert scheme of points on the K3 surface.

\subsubsection{Cubic fourfolds containing a plane}
This case was originally considered by
Voisin in her thesis~\cite{voisinthesis} and then by Hassett in \cite{hassett}. Let $X$ be a cubic fourfold in $\bP(V_6)$ that contains a plane
$\bP(V_3)$ for some $V_3\subset V_6$. This corresponds to the fact that the
Hodge structure of $X$ is in the Heegner divisor $\cD_8$ in the period domain.
As shown in~\cite{voisinthesis}, the blow up $\Bl_{\bP(V_3)}X$ projects onto
the plane $\bP^2=\bP(V_6/V_3)$ and the fibers are quadric surfaces which are
generically smooth. The discriminant locus, that is the locus where the
quadrics are singular, is a sextic curve in~$\bP^2$. The $2$-to-$1$ cover of $\bP^2$ ramified along the discriminant curve is a K3 surface $S$ of degree 2.

The variety $F\subset \Gr(2,V_6)$ of lines contained in $X$ is a hyperkähler
fourfold of $\KKK^{[2]}$-type by~\cite{bd}.
The lines in the fibers of $\Bl_{\bP(V_3)}X\to \bP^2$ form a uniruled divisor
$D$ in $F$. The divisor $D$ admits a $\bP^1$-fibration
over~$S$. In~\cite{voisinthesis} it was shown that the transcendental part
$H^2(F,\bZ)_\trans$ (of discriminant 8 and Hodge type $(1,19,1)$) embeds as a
sublattice of index two into the primitive cohomology $H^2(S,\bZ)_\prim$ (of
discriminant 2). This sublattice is closely related to the Brauer class $\beta$
induced by the $\bP^1$-fibration $D\to S$, so it should be considered as the
``primitive cohomology'' of the twisted K3 surface $(S,\beta)$
(see~\cite{vangeemen}). 

For a very general $X$ containing a plane the twisting class $\beta$ is non-trivial. Let $H$ be the hyperplane class in the cubic and $P$ the class of the plane. It was shown in \cite{hassett} that if there exists a $2$-dimensional cycle $T$ in $X$ whose intersection $TH^2-TP$ is odd then $\beta$ is trivial and $X$ is rational. Inside the divisor of cubics containing a plane, this condition occurs on a countable union of divisors on which the Fano variety of the cubic becomes birational to some space of (non-twisted!) moduli sheaves on a K3 surface; thus this case agrees with the conjecture on rationality of cubics mentioned before.

\subsection{Review of Debarre-Voisin varieties}

In \cite{dv} a maximal dimensional family of polarized hyperkähler fourfolds parametrized by $$\cM\coloneqq \bP\big(\bw3V_{10}^\vee\big)\git \SL(V_{10})$$ is constructed. Let $\sigma\in \wedge^3 V_{10}^\vee$ be a general trivector in ten variables. Then the zero locus
\[
X_6^\sigma\coloneqq \{[U_6]\in \Gr(6,V_{10}) \mid \sigma|_{U_6}=0\} \subset \Gr(6,V_{10})
\]
is a hyperkähler fourfold of type $K3^{[2]}$ with a natural polarization $\cO(1)|_{X_6^\sigma}$ of degree $22$ (with respect to the Beauville-Bogomolov form). From $\sigma$ one can also construct a hyperplane section $X_3^\sigma \subset \Gr(3,V_{10})$ and in \cite{dv} it was shown that there exists an isomorphism of Hodge structures $H^{20}(X_3^\sigma,\QQ)_{\van}\cong H^2(X_6^\sigma,\QQ)_{\prim}$, where $\van$ and $\prim$ denote the vanishing and primitive cohomology respectively. This isomorphism was extended in \cite{BenSong} to integer coefficients and to the Peskine sixfold
\[
X_1^\sigma\coloneqq \{[U_1]\in \bP(V_{10}) \mid \rank (\sigma(U_1,V_{10},V_{10}))\leq 6\},
\]
where $\sigma(U_1,V_{10},V_{10}) \in U_1^\vee \otimes \wedge^2 (V_{10}/U_1)^\vee$ should be thought of as a $9\times 9$ skew-symmetric matrix. More precisely, it was proven that $H^6(X_1^\sigma,\ZZ)_\van$, endowed with the intersection form on $X_1^\sigma$, and $H^2(X_6^\sigma,\ZZ)_\prim$, endowed with the Beauville-Bogomolov form, are isomorphic lattices modulo a Tate twist.

The Peskine sixfold is a codimension-3 pfaffian (which, in particular, is an orbital degeneracy locus, see \cite{benedetti}) associated to $9\times 9$ skew-symmetric matrices of rank $\leq 6$ and to the section $\sigma\in \wedge^3 V_{10}^\vee \cong H^0(\bP(V_{10},\wedge^2 \cQ^\vee (1))$. The same trivector $\sigma \in \wedge^3 V_{10}^\vee$ can be seen as a section of some vector bundle on other Grassmannians; doing so, it is possible to construct other varieties related to $X_1^\sigma$ as orbital degeneracy loci associated to $\sigma$. For instance, in \cite{BenSong} the Fano variety of lines contained in $X_1^\sigma$ has been constructed for general $\sigma$ as a degeneracy locus 
$X^\sigma_7 \subset  \Gr(7,V_{10})$. In the following sections we will study the rationality of $X_1^\sigma$ when $\sigma$ belongs to some divisors inside $\wedge^3 V_{10}^\vee$. These divisors are already well-known since they already appear in \cite{dv} and \cite{BenSong} (we refer to these two papers for background). 

\section{Rationality of the Peskine sixfold for $\sigma$ in $\cD^{3,3,10}$}
\label{sec:D22}

The divisor $\cD^{3,3,10}\subset \wedge^3 V_{10}^\vee$ is defined as the locus
\[
\cD^{3,3,10}\coloneqq \{\sigma\in \wedge^3 V_{10}^\vee \mid \exists V_3\subset V_{10}\mbox{ s.t. }\sigma(V_3,V_3,V_{10})=0 \}. 
\]
For $\sigma$ general in this divisor $X_3^\sigma$ is singular at $[V_3]$ and $X_6^\sigma$ is singular along a K3 surface $S$ of degree $22$. The converse holds as well (by \cite[Proposition 2.4]{BenSong}): if $X_3^\sigma$ or $X_6^\sigma$ are singular then $\sigma\in \cD^{3,3,10}$. Moreover $X_6^\sigma$ is birational to $S^{[2]}$. When it comes to rationality of $X_1^\sigma$, we can prove the following result.

\begin{theorem}
\label{thm_rationality_3310}
If $\sigma\in \cD^{3,3,10}$ is a general element then $X_1^\sigma$ is rational.
\end{theorem}

\begin{proof}
$\sigma\in \cD^{3,3,10}$ means that there exists $V_3\subset V_{10}$ such that $\sigma(V_3,V_3,V_{10})=0$. Let us project $X_1^\sigma$ to $\bP(V_{10}/V_3)\cong \bP^6$. Let $[V_4/V_3]\in \bP(V_{10}/V_3) $ be a general point and consider its fiber $\bP(V_4)\cap X_1^\sigma$. We will show that this fiber is a single point, away from $\bP(V_3)$ (where the projection is not defined). Let $l\in \bP(V_4)\setminus \bP(V_3)$. Then since $V_4/V_3$ is general, $\sigma(l,\cdot,\cdot)$ sends $V_3$ injectively to $V_{10}^\vee$; as a consequence, since $\sigma(l,V_4,V_4)=0$, there exists a unique $V_7\subset V_{10}$ such that $V_4\subset V_7$ and $\sigma(l,V_4,V_7)=0$. 

Notice that for any rank $6$ $2$-form $\omega$ in ten variables, and for any four dimensional $\omega$-isotropic subspace, there exists a seven dimensional $\omega$-isotropic subspace containing it and the kernel of $\omega$. Thus if $\sigma(l,\cdot,\cdot)$ has rank six, since $V_4$ is isotropic, there exists an isotropic seven dimensional subspace containing $V_4$. This subspace must be $V_7$ because $V_7$ is uniquely defined by the condition that $\sigma(l,V_4,V_7)=0$. So the fiber $(\bP(V_4)\cap X_1^\sigma)\setminus \bP(V_3)$ is the zero locus of $\sigma(l,\cdot,\cdot)|_{V_7/V_4}$, which is a section of the vector bundle $\wedge^2 (V_7/V_4)^\vee (1)\cong \cO(1)^3$ over $\bP(V_4)\setminus \bP(V_3)\cong \AA^3$. Such a zero locus is a linear subspace in $\bP^3$ and thus also in $\AA^3$. Since $X_1^\sigma$, being at least six dimensional, is not contained inside $\bP(V_3)$, it is non-empty for a general point $[V_4/V_3]$. We want to show that it is zero-dimensional for a general point $[V_4/V_3]\in \bP(V_{10}/V_3)$. 

Let us therefore consider a point $[V_4/V_3]\in \bP(V_{10}/V_3)$ such that the fiber is at least one dimensional. This means that there exists $V_2\subset V_4\subset V_7$ such that $\sigma(V_2,V_7,V_7)=0$. Since $\sigma(V_4,V_4,V_7)=0$ we deduce that there exists $U_6$ such that $V_4\subset U_6\subset V_7$ and $\sigma|_{U_6}=0$: indeed, such subspaces $U_6$ are defined by the condition $\sigma(V_4/V_2,U_6/V_4,U_6/V_4)=0$, which corresponds to a zero locus of $(V_4/V_2)^\vee\otimes \cO(1)\cong \cO(1)^{\oplus 2}$ inside $\Gr(2,V_7/V_4)\cong \bP^2$, and which is therefore never empty. Now, since $U_6$ contains $V_3$ and $[U_6]\in X_6^\sigma$, from \cite{dv} we know that $[U_6]\in S$, where $S$ is a K3 surface of degree $22$ defined by $$S:=\{ [U_6]\in X_6^\sigma \mid V_3\subset U_6 \} .$$
Let $\cU_3$ be the rank three vector bundle over $S$ whose fiber over $[U_6]\in S$ is $U_6/V_3$. Since $V_4/V_3 \subset U_6/V_3$ we deduce that $[V_4/V_3]\in \bP(\cU_3)$, with $\dim(\bP(\cU_3))=4$. 

Let $Z\subset \bP(V_{10}/V_3)$ be the locus of points over which the fiber of $X_1^\sigma \to \bP(V_{10}/V_3)\cong \bP^6$ is at least one dimensional. From what we have said $Z\subset \bP(\cU_3)$ and thus $\dim(Z)\leq 4$. But then, if the image of $X_1^\sigma \to \bP(V_{10}/V_3)\cong \bP^6$ is contained in $Z$ we deduce that the general fiber is at least two dimensional. However this is not possible since the existence of a two dimensional fiber translates into the existence of a subspace $V_3\subset V_7$ such that $\sigma(V_3,V_7,V_7)=0$. If such a subspace exists one can easily show that there exists $U_4$ such that $U_4\subset V_7$ and $\sigma(U_4,V_7,V_7)=0$. This means that $\sigma\in \cD^{4,7,7}\cap \cD^{3,3,10}$, where $$ \cD^{4,7,7}:=\{\sigma\in \wedge^3 V_{10}^\vee \mid \exists U_4\subset V_7\subset V_{10} \mbox{ s.t. } \sigma(U_4,V_7,V_7)=0\}.$$
The space $\cD^{4,7,7}$ is a divisor in $\wedge^3 V_{10}^\vee$ which is not contained in the divisor $\cD^{3,3,10}$ (see \cite[Proposition 3.3]{BenSong}), so for a general $\sigma\in\cD^{3,3,10}$ we have $\sigma\notin \cD^{4,7,7}$. Thus the general fiber of the rational map $X_1^\sigma \to \bP(V_{10}/V_3)\cong \bP^6$ must be zero dimensional and linear, hence a single point. We deduce that $X_1^\sigma \to \bP(V_{10}/V_3)\cong \bP^6$ is dominant and generically $1:1$, i.e. birational.
\end{proof}

\section{Rationality of the Peskine sixfold for $\sigma$ in $\cD^{1,6,10}$}
\label{sec:D24}

In the moduli space~$\wedge^3 V_{10}^\vee$ of trivectors, we have also the
divisor~$\cD^{1,6,10}$ given by 
$$
\cD^{1,6,10}\coloneqq \{\sigma\in \wedge^3 V_{10}^\vee \mid \exists V_1\subset V_6\subset V_{10}\mbox{ s.t. }\sigma(V_1,V_6,V_{10})=0 \}. 
$$
$\cD^{1,6,10}$ is also the locus
where the Peskine variety~$X_1^\sigma$ becomes singular and generically admits
an isolated singularity at~$[V_1]$. In this section we will study the
geometry along this divisor.  Notably we will recall the construction
of a K3 surface~$S$ of degree~$6$ and a divisor~$D$ in~$X^\sigma_6$ ruled
over~$S$, which was given in \cite{BenSong}. 

\subsection{Review of \cite{BenSong}}
\label{secconstrK3surf}

From now on, we let $\sigma\in\bw3 V_{10}^\vee$ be a general trivector in the
divisor $\cD^{1,6,10}$, so there is a unique distinguished flag $[V_1\subset V_6]$
such that $\sigma(V_1,V_6,V_{10})=0$. We will show how to construct a K3 surface $S$ of degree $6$
and a uniruled divisor $D$ in $X^\sigma_6$ that admits a $\bP^1$-fibration
over~$S$. The $\bP^1$-fibration
defines a non-trivial Brauer class $\beta\in\Br(S)$, and we know that
$X^\sigma_6$ can be recovered as a moduli space of $\beta$-twisted sheaves on
$S$.

\begin{proposition}\cite[Proposition 4.5]{BenSong}
\label{propK3S}
Suppose that the trivector $\sigma\in\bw3 V_{10}^\vee$ is general in the divisor
$\cD^{1,6,10}$, that is, we have $\sigma(V_1,V_6,V_{10})=0$ for a unique flag $V_1\subset
V_6\subset V_{10}$. Then $\sigma$ defines a smooth K3 surface $S$ of degree $6$ inside
$\Gr(2,V_{10}/V_6)$ as
\begin{equation}
\label{eqdefS1}
S=\setmid{[U_8]\in \Gr(2,V_{10}/V_6)\mbox{ s.t. }\exists U_4\subset U_8}{\begin{array}{c}V_1\subset U_4,\ \sigma(V_1,U_8,U_8)=0\\
\text{and }\sigma(U_4,U_4,U_8)=0\end{array}}
\end{equation}
\[
\subset \Gr(2,V_{10}/V_6).
\]
\end{proposition}

Next, one can construct a uniruled divisor $D$ in $X^\sigma_6$.

\begin{proposition}\cite[Proposition 4.7 and 4.8]{BenSong}
\label{prop:uniruled_divisor}
For a general~$\sigma$ in the divisor~$\cD^{1,6,10}$,
the set
\[D\coloneqq\setmid{[U_6]\in X^\sigma_6}{\exists [V_4\subset V_8]\in S\quad
V_4\subset U_6\subset V_8}\]
defines a divisor in $X^\sigma_6$ which has a smooth conic fibration $\pi\colon
D\to S$ over the K3 surface $S$. Moreover $[V_6]\in X_6^\sigma$ belongs to $D$ if and only if $V_1\subset V_6$.
\end{proposition}

In \cite{BenSong} it was shown that the Brauer class associated with the $\bP^1$-fibration $D\to S$ defines a twisted K3 surface $(S,\beta)$. Then one can prove that $X_6^\sigma$ is isomorphic to a moduli space of twisted sheaves on $(S,\beta)$.

\begin{proposition}
\label{prop_trivial_twisting}
If there exists a $2$-dimensional cycle inside $D$ whose intersection with a general fiber $l$ of $D\to S$ is odd then the twisting of $(S,\beta)$ is trivial.
\end{proposition}

\begin{proof}
The projection $D\to S$ is actually a smooth conic fibration (by construction), so the result follows by criteria of rationality of conic fibrations. Indeed, rationality of $D$ over the fraction field $k_S$ of $S$ is equivalent to the existence of a $k_S$-rational point. This point corresponds to the class of a relative $\cO(1)$ over $D\to S$, thus showing that the relative bundle $\cO(2)$ is in fact divisible, i.e. that $\beta$ is trivial in the Brauer group. In order for the class of $\cO(1)$ to exist, since $\cO(2)=K_{D/S}$ is the relative canonical bundle, it is sufficient that the class of $\cO(d)$ exists for $d$ odd, which is ensured by the hypothesis in the statement.
\end{proof}

\subsection{Birational models of the Peskine sixfold}
\label{sec_bir_models}

In this section we construct two birational models $X_{1.1}^\sigma$ and $X_{1.2}^\sigma$ of the Peskine sixfold. The latter will become important later, as it can be described in a very explicit way and its rationality can be understood quite well. Let us start with a lemma.

\begin{lemma}
\label{lemma_cubic}
The intersection $\cC:=X_1^\sigma \cap \bP(V_6)$ is a cubic fourfold.
\end{lemma}

\begin{proof}
Let us look at the equations that define $\cC$. Let us consider a point $[U_1]\in \cC$, and let us suppose that $[U_1]\neq[V_1]$. 
Since $U_1\subset V_6$ we have $\sigma(U_1,U_1+V_1,V_{10})=0$, i.e. $U_1+V_1\in \ker \sigma(U_1,\cdot,\cdot)$. Thus $[U_1]\mapsto \sigma(U_1,\cdot,\cdot)|_{V_{10}/(U_1+V_1)}$ defines a section $s$ of $\wedge^2 \cO(1)\otimes(V_{10}/(\cO(-1)+V_1))^\vee$ over $\bP(V_6)\setminus [V_1]$. Clearly $\sigma(U_1,\cdot,\cdot)$ has rank $\leq 6$ if and only if $s(U_1)$ has rank $\leq 6$ as a skew-symmetric matrix of size $8\times 8$. 
Thus its Pfaffian $\Pf(s)$ is well-defined over $\bP(V_6)\setminus [V_1]$ and can be extended to $\bP(V_6)$ by Hartogs' Lemma. We deduce that $\cC=\{\Pf(s)=0\}$ is a hypersurface in $\bP(V_6)$. 
Over $\bP(V_6)\setminus [V_1]$, the element $\Pf(s)$ is a section of 
$\det((V_{10}/(\cO(-1)+V_1))^\vee)\otimes \cO(8/2)=\det((V_{10}/\cO(-1))^\vee)\otimes V_1\otimes \cO(4)|_{\bP(V_6)\setminus[V_1]}\cong \cO(3)|_{\bP(V_6)\setminus[V_1]}$, thus its extension is a section of $\cO(3)$ over $\bP(V_6)$.
\end{proof}

\begin{remark}
We suspect that the cubic is singular at $[V_1]$. Indeed, the fact that from an element $\sigma\in \cD^{1,6,10}$ we recover a singular cubic fourfold would not be surprising. We have seen that a general $\sigma\in \cD^{1,6,10}$ defines a general $K3$ surface $S$ of degree $6$, and it is well known that to a general singular cubic fourfold one can associate a $K3$ surface $S'$ of degree $6$ such that the Fano variety of lines of the cubic is birational to the Hilbert scheme $S'^{[2]}$. It is then natural to ask the question whether the Fano variety of lines of $\cC$ is birational to $S^{[2]}$. 
\end{remark}

Let us consider the projection $X_1^\sigma \dashrightarrow \bP(V_{10}/V_6)\cong \bP^3$. This is not defined at $\cC= X_1^\sigma \cap\bP(V_6)$. We want to understand what is the closure of the general fiber of a point of this rational map, outside of the indeterminacy locus. Therefore we consider a general point in $\bP(V_{10}/V_6)$, which corresponds to a subspace $U_7$ such that $V_6\subset U_7\subset V_{10}$. Let us denote by $\sigma_l:=\sigma(l,\cdot,\cdot)|_{V_{10}/l}$ for $l\in \bP(V_{10})$. The closure $T_{U_7}$ of the fiber (outside of $\cC$) over $[U_7]$ of the rational map is the closure of
\[
T^0_{U_7}:=\{ [l]\in \bP(U_7) \mid l\notin \bP(V_6)\mbox{ and }\rank(\sigma_l)\leq 6 \}.
\]

\begin{lemma}
A general trivector $\sigma\in \cD^{1,6,10}$ defines a non-degenerate two-form $\omega:=\sigma_{V_1}|_{V_{10}/V_6}\in \wedge^2 (V_{10}/V_6)^\vee \cong H^0(\bP(V_{10}/V_6), \cQ^\vee(1))$, and an injection $\cO(-1) \subset \Omega(1) \subset (V_{10}/V_6)\otimes \cO$ such that $(V_{10}/V_6)/\Omega(1) \cong \cO(1)$.
\end{lemma}

\begin{proof}
For a general $\sigma\in \cD^{1,6,10}$, $\sigma(V_1,\cdot,\cdot)$ is a rank four two-form on $V_{10}/V_6$, thus $\omega$ is non-degenerate. The bundle $\Omega(1)$ is defined fiberwise as follows: for any point $p\in \bP(V_{10}/V_6)$, $\Omega(1)_p:=\{ v\in V_{10}/V_6 \mid w(p,v)=0\}=p^\perp \cong \CC^3$. Clearly $p\subset \Omega(1)_p$ and $(V_{10}/V_6)/\Omega(1)_p \cong p^\vee$, where the isomorphism is given by $w$ itself.
\end{proof}

\begin{remark}
Since $\Omega(1)_p$ can be identified with $p^\perp$, we will denote it by $\cO(-1)^\perp \cong \Omega(1)$, where $\perp$ is with respect to $\omega$, in order to make the dependance from $\omega$ explicit. Moreover notice that $\cN:=\cO(-1)^\perp / \cO(-1)$ is a rank-$2$ null-correlation bundle.
\end{remark}

From the previous lemma, for any $U_7\in \bP(V_{10}/V_6)$ one can consider a nine-dimensional subspace $U_7^\perp:=\cO(-1)^\perp_{U_7}+V_6 \subset V_{10}$. By letting $U_7\in \bP(V_{10}/V_6)$ vary, one obtains a rank-nine vector bundle $\cV_9\subset V_{10}\otimes \cO$ which fits in a short exact sequence
\[
0\to V_6\otimes\cO \to \cV_9 \to \cO(-1)^\perp \to 0.
\]
Notice that, by construction, $V_1\subset \ker \sigma|_{U_7^\perp}$. Since $\cV_9$ is a subbundle of the trivial bundle $V_{10}\otimes \cO$, for any $U_7\in\bP(V_{10}/V_6)$, the section $\sigma$ defines a section $$\sigma'_{U_7}\in H^0 (\bP(U_7)\setminus [V_1], \wedge^2 (U_7^\perp/\cO(-1)+V_1)^\vee(1));$$ indeed, for any $l\in \bP(U_7)$, $l\neq V_1$, set $\sigma'_{U_7}(l):=\sigma_l|_{U_7^\perp/l+V_1}$.

\begin{lemma}
\label{lemma_fiber_T}
Let $U_7$ be a point in $\bP(V_{10}/V_6)$, and let $l\in \bP(U_7)\setminus \cC$. Then $l\in T_{U_7}^0$ if and only if $\rank(\sigma'_{U_7}(l))= 4$.
\end{lemma}

\begin{proof}
Let $0\neq v\in (V_{10}/ U_7^\perp)^\vee$. If $l$ is such that $\rank(\sigma'_{U_7}(l))\leq 4$, we can write $\sigma_l=\tilde{\sigma'_{l}} + v\wedge v'$ for $v'\in V_{10}^\vee$, where $\tilde{\sigma'_{l}}$ is a lift of $\sigma_{l}|_{U_7^\perp}$ to $\wedge^2 (V_{10}/l)^\vee$. Thus $\tilde{\sigma'_{l}}$ has rank at most four and $\rank(\sigma_{l})\leq 6$, proving that $l\in X_1^\sigma$.

Conversely, assume that $l\in T_{U_7}^0$, i.e. $l\in X_1^\sigma$. This means that $\rank(\sigma_l)= 6$. Let $K:=\ker \sigma_l$ be the four-dimensional kernel in $V_{10}$. Since $l\notin \bP(V_6)$ we have that $V_1\cap K=\{0\}$, and clearly $K\subset U_7^\perp$. Since $V_1\subset \ker \sigma|_{U_7^\perp}$ we deduce that $K+V_1\subset \ker \sigma_l|_{U_7^\perp}$. Since $\dim(K+V_1)=5$ and $\dim(U_7^\perp)=9$, we deduce that $\sigma_l|_{U_7^\perp / l+V_1}=\sigma'_{U_7}(l)$ has rank at most four.

Finally, following the same line of ideas, one can show that $\rank(\sigma'_{U_7}(l))<4$ never happens for $l\in \bP(U_7)\setminus \cC$ because otherwise $\rank(\sigma_l)\leq 4$, which is a contradiction with the fact that $[V_1]\in X_1^\sigma$ is the unique point such that $\rank(\sigma_{[V_1]})\leq 4$.
\end{proof}

Let us denote by $\cV_7$ the rank seven vector bundle over $\bP(V_{10}/V_6)\cong \bP^3$, whose fiber over $[U_7]$ is $U_7$. This bundle fits in the two following short exact sequences:
\[
0\to V_6\otimes \cO \to \cV_7 \to \cO(-1)\to 0,
\]
\[
0\to \cV_7 \to \cV_9 \to \cN \to 0.
\] 

$\cV_7$ is a bundle on $\bP^3$ containing the bundle $V_6\otimes \cO$, and thus containing $V_1\otimes \cO$ as well. Thus (by abuse of notation) $\cV_7/V_1$ is a rank 6 vector bundle, and one can consider the associated projective bundle $\bP(\cV_7/V_1)\to \bP(V_{10}/V_6)\cong \bP^3$ parametrizing, inside a fiber, 2-dimensional subspaces containing $V_1$. Let $\cU_1$ be the relative rank one tautological bundle over the projective bundle $\bP(\cV_7/V_1)$. In order to obtain a birational model of $X_1^\sigma$ we need a further step in the construction. Indeed, the fiber of $\bP(\cV_7/V_1)$ over $[U_7]\in \bP(V_{10}/V_6)$ is $\bP(U_7/V_1)$ and a point of such a fiber is a subspace $U_2$ such that $V_1\subset U_2\subset U_7$. Similarly to before, we consider the rank two vector bundle $\cV_2$ over $\bP(\cV_7/V_1)$ whose fiber over $[U_2\subset U_7]$ is $U_2$. Thus we have a well defined tower of projective bundles $\bP(\cV_2)\to \bP(\cV_7/V_1) \to \bP(V_{10}/V_6)$ of dimension $9$; let us denote by $\cO_r(-1)$ the last relative tautological bundle. We have a natural birational projection morphism $\bP(\cV_2)\to \bP(V_{10})\cong \bP^9$. 

A point in $\bP(V_{10})$ is a one-dimensional subspace $l\subset V_{10}$. Let us suppose that $l$ is not contained inside $\bP(V_6)$. Then the preimage of $l$ via $\bP(\cV_2)\to \bP(V_{10})\cong \bP^9$ inside $\bP(\cV_2)$ is given by a unique point, i.e. the flag $[l\subset U_2\subset U_7]$ where $U_2=l+V_1$ and $U_7=l+V_6$. 
\begin{definition}
Let $X_{1.1}^\sigma$ be the closure inside $\bP(\cV_2)$ of the inverse image of $X_1^\sigma\setminus \bP(V_6)$ via $\bP(\cV_2)\to \bP(V_{10})$.
\end{definition}
Since $X_1^\sigma\cap \bP(V_6)=\cC$ is a proper subvariety of $X_1^\sigma$ and $\bP(\cV_2)\to \bP(V_{10})$ is a bijection outside of $\bP(V_6)$, we automatically deduce the following.

\begin{proposition}
\label{prop_birational_1.1}
The projection $\bP(\cV_2)\to \bP(V_{10})$ restricts to a birational morphism $X_{1.1}^\sigma\to X_1^\sigma$.
\end{proposition} 

Notice that there exists a morphism $X_{1.1}^\sigma\to \bP(V_{10}/V_6)$, which is given by restricting the projection $\bP(\cV_2)\to \bP(\cV_7/V_1)\to \bP(V_{10}/V_6)$. In the following of this section we will try to understand better $X_{1.1}^\sigma$. Then, we will show that $X_{1.1}^\sigma$ projects birationally to a subvariety $X_{1.2}^\sigma$ inside $\bP(\cV_7/V_1)$. We will use this and the fact that $X_{1.2}^\sigma$ has a quite simple description to address the question of its rationality and thus the rationality of $X_1^\sigma$. 

Let us denote by $$Y:=\{[l\subset U_2\subset U_7]\in \bP(\cV_2) \mid l\in \bP(V_6)\}\subset \bP(\cV_2).$$
First, notice that $\sigma$ defines a section 
$$\sigma'\in H^0(\bP(\cV_2)\setminus Y,\cO_r(1)\otimes \wedge^2 (\cV_9/\cV_2)^\vee) \mbox{ , }[l\subset U_2\subset U_7]\mapsto \sigma'_{U_7}(l)=\sigma_l|_{U_7^\perp/l+V_1}.$$

\begin{lemma}
\label{lem_closure_1.1}
The variety $X_{1.1}^\sigma$ can be described as the closure of
\[
(X_{1.1}^\sigma)^0:=\{ [l\subset V_2\subset V_7]\in \bP(\cV_2)\setminus Y\mid  \rank (\sigma'(l))= 4\}.
\]
It is a codimension three subvariety of $\bP(\cV_2)$.
\end{lemma}

\begin{proof}
This is just a global version of Lemma \ref{lemma_fiber_T}. The second assertion is a consequence of Proposition \ref{prop_birational_1.1}; however, it can also be deduced by the fact that the condition $\rank(\cdot)\leq 4$ is a codimension three condition in the space of skew-symmetric matrices of size $7\times 7$. 
\end{proof}

\begin{definition}
Let us denote by $X_{1.2}^\sigma\subset \bP(\cV_7/V_1)$ the reduced image of $X_{1.1}^\sigma$ via the projection $\bP(\cV_2)\to \bP(\cV_7/V_1)$. 
\end{definition}

As already shown, $X_{1.2}^\sigma$ projects to $\bP(V_{10}/V_6)$. We want to describe a general fiber of this morphism, and more generally we want to understand better $X_{1.2}^\sigma$; in order to do so, we will need a small digression to study $7\times 7$ skew-symmetric matrices. For the moment, let us resume the constructions we have defined in the following diagram.

\[\begin{tikzcd}
& X_1^\sigma \ar[r,hook] & \bP(V_{10})\\
X_{1.1}^\sigma\ar[r,hook]\ar[rd,"\mathrm{}"'] \ar[ru,"\mathrm{birat.}"]& \bP(\cV_2)\ar[rd,"\bP^1"'] \ar[ru,"\mathrm{birat.}"] \\
& X_{1.2}^\sigma\ar[r,hook]& \bP(\cV_7/V_1)\ar[r,"\bP^5"]
&\bP(V_{10}/V_6).
\end{tikzcd}\]

The only missing point in this diagram is the birationality between $X_{1.1}^\sigma$ and $X_{1.2}^\sigma$, for which we still have to work a little bit.

\subsubsection{Digression: on skew-symmetric matrices in seven variables}
\label{sec_digression_matrices}

Let $A_2\subset A_7$ be two vector spaces of dimension two and seven respectively. We want to study the vector space $\wedge^2 A_7$. More precisely, let us consider the twenty dimensional quotient $B:=(\wedge^2 A_7)/\wedge^2 A_2$. The space $B$ is an extension
\[
0\to A_2\otimes (A_7/A_2) \to B \to \wedge^2 (A_7/A_2)\to 0.
\]
We will denote by $b$ an element in $\wedge^2 A_7$, and by $\overline{b}$ its image in $B$. For any $[A_1]\in \bP(A_2)$, any element $b\in \wedge^2 A_7$ defines an element $(b \mod A_1)\in \wedge^2 (A_7/A_1)\cong \wedge^2 \CC^6$. Let $P\subset \GL(A_7)$ be the stabilizer of $A_2$. Let us define a $P$-invariant closed subvariety $O_2\subset B$ as
\[
O_2=\{\overline{b}\in B \mid \forall [A_1]\in \bP(A_2) \mbox{ , }\Pf(b \mod A_1) =0\},
\]
where $\Pf$ is the Pfaffian of skew-symmetric matrices (in this case of size $6\times 6$). Notice that $O_2$ is well defined since $\Pf(b \mod A_1)$ does not depend on the representative $b$, but only on the element $\overline{b}$. It is straightforward to check that $O_2$ is the intersection of two degree-three hypersurfaces in $B$ (seen as an affine space). More precisely, the projective line $\bP(A_2)\cong\bP^1$ defines a pencil of cubics given exactly by $[A_1] \mapsto \Pf(\cdot \mod A_1)\in \CC[\wedge^2(A_7/A_1)^\vee]_3\subset \CC[B^\vee]_3$, and $O_2$ is the intersection of all the cubic hypersurfaces defined by cubics in this pencil. With a little bit of help from \cite{m2} we can prove the following result.

\begin{lemma}
$O_2$ is a codimension two subvariety of $B$ given by the intersection of a pencil of cubic hypersurfaces. Its singular locus has codimension $5$ inside $B$ (and thus codimension $3$ inside $O_2$).
\end{lemma}

\begin{proof}
Starting from the explicit expression of the Pfaffian of a $6\times 6$ skew-symmetric matrix, one can recover the ideal of $O_2$ inside $\CC[B^\vee]$, which is generated by two degree-three polynomials. By using \cite{m2} one can compute the Jacobian of this ideal, then recover the ideal generated by the $2\times 2$ minors of the Jacobian, then compute its Krull dimension that turns out to be equal to $15$. It can also be checked that the Krull dimension of the ideal of $O_2$ has dimension equal to $18$.
\end{proof} 

Let us also define another $P$-invariant closed variety inside $B$. Consider the $P$-variety $\bP(A_2)\times \Gr(2,A_7/A_5)$. Denote by $\cO(-1)$ the tautological line bundle over $\bP(A_2)$ and by $\cU_2$ (respectively $\cQ_3$) the tautological rank-$2$ bundle (resp. rank-$3$ quotient bundle) on $\Gr(2,A_7/A_5)$. The bundle $(A_2/\cO(-1))\otimes \cU_2 + \cO(-1)\otimes (A_7/A_2) $ is a subbundle of $A_2\otimes (A_7/A_2)$. Moreover the bundle $\cW$ defined by the exact sequence
\[
0\to (A_2/\cO(-1))\otimes \cU_2 + \cO(-1)\otimes (A_7/A_2) \to \cW \to \wedge^2(\cU_2)\to 0 
\]
is a subbundle of the trivial bundle $B\otimes \cO_{\bP(A_2)\times \Gr(2,A_7/A_5)}$. Thus its total space projects to $B$; let us denote by $O_5$ the closure of its image.

\begin{lemma}
\label{lem_orbit_lines_throughV1}
The affine variety $O_5$ has codimension at least $5$ inside $B$. If $b$ has rank four in $\wedge^2 A_7$ and $(\overline{b} \mod A_2)$ has rank two in $ \wedge^2 (A_7/A_2)$, then $\overline{b}\in O_5$.
\end{lemma}

\begin{proof}
The total space of $\cW$ has dimension $15$ so its image can have at most dimension $15$ and since $\dim(B)=20$ we get that $\codim(O_5)\geq 5$. Consider a basis $a_1,\cdots,a_7$ of $A_7$ with $A_2=\langle a_1,a_2\rangle$. The hypothesis on $(\overline{b} \mod A_2)$ tells us that we can write $b$ as $a_1\wedge x + a_2 \wedge y + a_3\wedge a_4$. If $x\in \langle a_1,\cdots,a_4\rangle$ then $\overline{b}\in O_5$, thus we suppose that $x\notin \langle a_1,\cdots,a_4\rangle$. If $y\notin \langle a_1,\cdots,a_4,x\rangle$ then $b$ has rank six, which is excluded by our hypothesis. So $y=y'+ux$ with $y'\in \langle a_1,\cdots,a_4\rangle$ and $u\in \CC$. As a result $b=(a_1+ua_2)\wedge x + a_2\wedge y' + a_3\wedge a_4$ and we get that $\overline{b}\in O_5$.
\end{proof}

\subsection{Rationality of $X_{1,2}^\sigma$}

We start this section by showing that $\sigma$ defines a section $\sigma''$ which will completely determine $X_{1.2}^\sigma \subset \bP(\cV_7/V_1)$. Recall that over $\bP(\cV_7/V_1)$ we have a flag of bundles $\cV_2\subset \cV_7\subset \cV_9$. Let us define $E_B:=(\wedge^2 (\cV_9/\cV_2)^\vee/\wedge^2 (\cV_9/\cV_7)^\vee)\otimes \cU_1^\vee$; this bundle fits in a short exact sequence
\[
0\to (\cV_7/\cV_2)^\vee \otimes (\cV_9/\cV_7)^\vee \otimes \cU_1^\vee \to E_B \to \wedge^2 (\cV_7/\cV_2)^\vee \otimes \cU_1^\vee \to 0.
\]
Notice that the fiber $(E_B)_p$ of $E_B$ over a point $ p$ is naturally isomorphic to $B$.

\begin{lemma}
\label{lem_gg_EB}
An element $\sigma\in \cD^{1,6,10}$ defines a section $\sigma''\in H^0(\bP(\cV_7/V_1), E_B )$. Moreover the sections $\{\sigma''\}_{\sigma\in \cD^{1,6,10}}$ thus constructed globally generate $E_B$.
\end{lemma}

\begin{proof}
In order to construct $\sigma''$ we will show that for any point $p=[U_2\subset U_7]\in \bP(\cV_7/V_1)$, $\sigma$ defines an element $\sigma''(p)\in (U_2/V_1)^\vee\otimes (\wedge^2 (U_7^\perp/U_2)^\vee / \wedge^2(U_7^\perp/U_7)^\vee)$. By the canonicity of the construction it will be clear that $\{\sigma''(p)\}_{p\in \cP(\cV_7/V_1)}$ glue together to form a section of $E_B$. We already know that $\sigma(U_2,U_2,U_7^\perp)=0$ since $\sigma(V_1,U_7,U_7^\perp)=0$. This also implies that $\sigma_{V_1}|_{U_7^\perp}\in \wedge^2 (U_7^\perp/U_7)^\vee \subset \wedge^2(U_7^\perp/U_2)^\vee$. Therefore, when we quotient $\wedge^2(U_7^\perp/U_2)^\vee$ by $\wedge^2 (U_7^\perp/U_7)^\vee$, we obtain that $(\sigma_{V_1}|_{U_7^\perp} \mod \wedge^2 (U_7^\perp/U_7)^\vee) =0$. This implies that $(\sigma(U_2,U_7^\perp,U_7^\perp) \mod \wedge^2 (U_7^\perp/U_7)^\vee)$ is actually contained in $(U_2/V_1)^\vee\otimes (\wedge^2 (U_7^\perp/U_2)^\vee / \wedge^2(U_7^\perp/U_7)^\vee)$. Thus we define $\sigma''(p)$ to be equal to $(\sigma(U_2,U_7^\perp,U_7^\perp) \mod \wedge^2 (U_7^\perp/U_7)^\vee)$.

In order to show that the set of such sections globally generates $E_B$, we will show that for any point $p=[U_2\subset U_7]\in \bP(\cV_7/V_1)$, and any $x\in (U_2/V_1)^\vee\otimes (\wedge^2 (U_7^\perp/U_2)^\vee / \wedge^2(U_7^\perp/U_7)^\vee)$, we can find a section $\sigma$ such that $\sigma''(p)=x$. Let $e_1,\cdots,e_6$ be a basis of $V_6^\vee$, $e_0,e_1$ a basis of $U_2^\vee$, $e_0,\cdots,e_8$ a basis of $(U_7^\perp)^\vee$ and $e_0,\cdots,e_9$ a basis of $V_{10}^\vee$. With this notation $x=e_0\wedge x'$, where $x'\in (\wedge^2 \langle e_2,\cdots,e_8 \rangle / e_7\wedge e_8)$. Take any lift $\tilde{x'}$ of $x'$ in $\wedge^2 \langle e_2,\cdots,e_8 \rangle$. Then define $\sigma:= e_1 \wedge (e_0\wedge e_9 + e_7\wedge e_8) + e_0\wedge \tilde{x'}$, and it is clear that $\sigma\in \cD^{1,6,10}$ and $\sigma(p)=x$.
\end{proof}  

The first result will allow us to prove the birationality between $X_{1.1}^\sigma$ and $X_{1.2}^\sigma$.

\begin{lemma}
\label{lem_ODL_O5}
The orbital degeneracy locus
\[
D_{O_5}(\sigma''):=\{ p\in \bP(\cV_7/V_1) \mid \sigma''(p) \in O_5 \subset B\cong (E_B)_p \}
\]
has dimension at most three.
\end{lemma}

\begin{proof}
$D_{O_5}(\sigma'')$ is a degeneracy locus associated to the $P$-invariant subvariety $O_5\subset B$ and the section $\sigma''$. Thus, we can apply the Bertini type theorem \cite[Proposition 2.3]{benedetti} since by Lemma \ref{lem_gg_EB} $\sigma''$ is a general element in a linear system that globally generates $E_B$. Thus the codimension of $D_{O_5}(\sigma'')$ in $\bP(\cV_7/V_1)$ is equal to the codimension of $O_5$ in $B$, and the result follows from Lemma \ref{lem_orbit_lines_throughV1}.
\end{proof}

\begin{proposition}
\label{prop_birationality_1.2}
The morphism $X_{1.1}^\sigma\to X_{1.2}^\sigma$ is birational for $\sigma$ generic in $\cD^{1,6,10}$.
\end{proposition}

\begin{proof}
The birationality amounts to proving that for a general point $p=[l\subset U_2\subset U_7]\in X_{1.1}^\sigma$, the set $Z_p:=\{l' \subset U_2 \mid [l'\subset U_2 \subset U_7]\in X_{1.1}^\sigma\}$ consists of the single (reduced) point $p$. Let us restrict to $(X_{1.1}^\sigma)^0$, and let us suppose that $p\in (X_{1.1}^\sigma)^0$. Also, to prove birationality it will be sufficient to prove that $Z_p\cap (X_{1.1}^\sigma)^0=\{p\}$ because $(X_{1.1}^\sigma)^0$ is an open dense subset in $X_{1.1}^\sigma$. 

Therefore we can assume that $l\cap V_6=\{0\}$ and $U_2\cap V_6=V_1$. Let $0\neq v_1\in V_1$ and $0\neq v_0\in l$, and let us suppose that $l'=[av_1+bv_0]\in \bP(U_2)\setminus [V_1]$, i.e. $b\neq 0$. Let us also fix a basis $v_0,\cdots,v_8$ of $U_7^\perp$ such that $v_1,\cdots,v_6$ is a basis of $V_6$ (and $v_0\cdots,v_6$ is a basis of $U_7$). Let $e_0,\cdots,e_8$ be the dual basis of $(U_7^\perp)^\vee$. Without loss of generality, we can assume that $\sigma_{[V_1]}=e_7\wedge e_8$. Moreover $\sigma'(l)$ has rank four by Lemma \ref{lemma_fiber_T}, and $\sigma'(l')=a e_7\wedge e_8 + b\sigma'(l)\in \wedge^2 \langle e_2,\cdots,e_8 \rangle$. 

If $\rank(\sigma''([U_2\subset U_7]))\leq 2$ then by Lemma \ref{lem_orbit_lines_throughV1} we get that $[U_2\subset U_7]\in D_{O_5}(\sigma'')$. This implies that $\sigma'(l)=e_7\wedge x + y\wedge z$ and $\rank(\sigma'(l'))\leq 4$ for any $l'\in \bP(U_2)$, i.e. by Lemma \ref{lemma_fiber_T} that $\bP(U_2)$ is a line which contains $[V_1]$ and is entirely contained inside $X_1^\sigma$. However by Lemma \ref{lem_ODL_O5} the set of points $[l'\subset U_2\subset U_7]$ such that $[U_2\subset U_7]\in D_{O_5}(\sigma'')$ can at most have dimension $3+1=4$ and thus cannot cover $X_{1.1}^\sigma$. Thus, for a general point $p=[l\subset U_2\subset U_7]\in X_{1.1}^\sigma$ we can suppose that $\rank(\sigma''([U_2\subset U_7]))\geq 4$ and $\sigma'(l)=w\wedge x + y\wedge z$, with $\dim(\langle e_7,e_8,w,x,y,z \rangle)=6$. But then it is straightforward to see that $Z_p\cap (X_{1.1}^\sigma)^0=\{p\}$.
\end{proof}

The following is the diagram representing the birational models of $X_1^\sigma$:

\[\begin{tikzcd}
& X_{1.1}^\sigma \ar[ld,"\mathrm{birat.}"] \ar[r,"\mathrm{birat.}"] \ar[rd,"\mathrm{birat.}"'] & \Bl_{[V_1]}X_1^\sigma \ar[d,"\mathrm{birat.}"] \\
X_{1.2}^\sigma && X_1^\sigma.
\end{tikzcd}\]

The map $X_{1.1}^\sigma \to \Bl_{[V_1]}X_1^\sigma$ comes from the fact that $\Bl_{[V_1]}\bP(V_{10})$ can be seen as the total space of the $\bP^1$-bundle $\bP(V_1+\cO(-1))$ over $\bP(V_{10}/V_1)$; the image of the map $X_{1.1}^\sigma \to \bP(V_1+\cO(-1))$ is contained in $\Bl_{[V_1]}X_1^\sigma$ and gives us the morphism represented in the previous diagram. 

We are ready now to look more closely at the equations defining $X_{1.2}^\sigma$.

\begin{proposition}
\label{prop_description_1.2_quadrics}
For $\sigma\in \cD^{1,6,10}$ general, $X_{1.2}^\sigma$ is a six-dimensional irreducible variety isomorphic to the orbital degeneracy locus 
\[
D_{O_2}(\sigma''):=\{ p\in \bP(\cV_7/V_1) \mid \sigma''(p) \in O_2 \subset B\cong (E_B)_p \}.
\]
$X_{1.2}^\sigma$ is singular in codimension three and projects dominantly to $\bP(V_{10}/V_6)\cong \bP^3$; the fiber of this morphism over a point $[U_7]\in \bP(V_{10}/V_6)$ is the intersection of a pencil of (generically smooth) quadrics inside $\bP(U_7/V_1)\cong \bP^5$.
\end{proposition}

\begin{proof}
Let us start by showing that the second statement holds true for $D_{O_2}(\sigma'')$, then we will prove that $X_{1.2}^\sigma\cong D_{O_2}(\sigma'')$. As already said, $D_{O_2}(\sigma'')$ is a degeneracy locus associated to the $P$-invariant subvariety $O_2\subset B$ and the section $\sigma''$. Thus, we can apply the Bertini type theorem \cite[Proposition 2.3]{benedetti} since by Lemma \ref{lem_gg_EB} $\sigma''$ is a general element in a linear system that globally generates $E_B$. Thus $D_{O_2}(\sigma'')$ is six-dimensional, its singular locus $\Sing D_{O_2}(\sigma'')$ is equal to
\[
D_{\Sing (O_2)}(\sigma''):= \{ p\in \bP(\cV_7/V_1) \mid \sigma''(p) \in \Sing (O_2) \subset B\cong (E_B)_p \}
\] 
and $\codim_{D_{O_2}(\sigma'')}(D_{\Sing (O_2)}(\sigma''))= \codim_{O_2}(\Sing(O_2))=3$. From the description of $O_2$ in Section \ref{sec_digression_matrices} it turns out that $D_{O_2}(\sigma'')$ is the intersection of a pencil of hypersurfaces. In order to see it, let us look at the fiber $D_{U_7}$ of $D_{O_2}(\sigma'')\to \bP(V_{10}/V_6)$ over a point $[U_7]\in \bP(V_{10}/V_6)$. Then $D_{U_7}\subset \bP(U_7/V_1)\cong\bP^5$ and it is defined as follows. For any seven dimensional subspace $W$ such that $U_7/V_1 \subset W \subset U_7^\perp/V_1$ one gets a section $\Pf(\sigma''|_{W})$ of the bundle $\det(W/U_2)^\vee\otimes (\cU_1^\vee)^3$. Since $W$ is a trivial extension of the dual of the quotient bundle over $\bP^5$ by a trivial factor $W/U_7\otimes \cO$, we get that $\Pf(\sigma''|_{W})$ is a section of $\det(\cQ^\vee)\otimes (\cU_1^\vee)^3=(\cU_1^\vee)^2=\cO_{\bP^5}(2)$, i.e. a quadric in $\bP^5$. 

By letting $W$ vary in $\bP(U_7^\perp/U_7)\cong \bP^1$, we obtain a pencil of quadrics whose intersection is exactly the fiber $D_{U_7}$. Since such a zero locus is never empty, $D_{O_2}(\sigma'')$ surjects onto $\bP(V_{10}/V_6)$. Moreover the general quadric in the pencil for a general point $[U_7]$ is smooth since it is the Pfaffian of a general two-form in six variables, and it is well-known that such a quadric is singular in codimension six. 

In order to show that $X_{1.2}^\sigma\cong D_{O_2}(\sigma'') $ we will show that they are the closure of the same subvariety in $\bP(\cV_7/V_1)$. More precisely, let $(X_{1.2}^\sigma)^0$ be the image of $(X_{1.1}^\sigma)^0$ under the birational morphism $X_{1.1}^\sigma\to X_{1.2}^\sigma$ (the variety $(X_{1.1}^\sigma)^0$ was defined in Lemma \ref{lem_closure_1.1}). If we denote by $$Y_2=\{[U_2\subset U_7]\in \bP(\cV_7/V_1)\mid U_2\subset V_6\}$$
then $(X_{1.2}^\sigma)^0=X_{1.2}^\sigma\setminus Y_2$. We will show that $$D_{O_2}(\sigma'')^0:=D_{O_2}(\sigma'')\setminus Y_2$$
is exactly the same as $(X_{1.2}^\sigma)^0$ and that it is a dense open of $D_{O_2}(\sigma'')$.

Firstly, $D_{O_2}(\sigma'')^0$ is dense in $D_{O_2}(\sigma'')$ because $D_{O_2}(\sigma'')\cap Y_2$ can again be described inside $Y_2$ as a degeneracy locus associated to $O_2\subset B$ and a general section; thus it has codimension two inside $Y_2\cong\bP(V_6/V_1)\times \bP(V_{10}/V_6)$, which has dimension seven. As a consequence $D_{O_2}(\sigma'')\cap Y_2$ has dimension five and $D_{O_2}(\sigma'')\setminus Y_2$ is dense in $D_{O_2}(\sigma'')$.

Now, let us show that $D_{O_2}(\sigma'')^0$ and $(X_{1.2}^\sigma)^0$ describe the same set in $\bP(\cV_7/V_1)$; since to both we attach the reduced scheme structure, the isomorphism $D_{O_2}(\sigma'') \cong X_{1.2}^\sigma$ will follow. Let $[l\subset U_2\subset U_7]\in (X_{1.1}^\sigma)^0$, which means that $\rank(\sigma'(l))=4$. This by definition of $\sigma'$ means that $\sigma_l|_{U_7^\perp}$ has rank four. In the basis $e_0,\cdots,e_9$ appearing in the proof of Lemma \ref{lem_gg_EB} we also know that $\sigma_{V_1}|_{U_7^\perp}=e_7\wedge e_8$. Thus we deduce that $\sigma_{av_1+bv_0}|_{U_7^\perp}=ae_7\wedge e_8 + b \sigma_l|_{U_7^\perp}$. When we mod out by $\CC e_7\wedge e_8 \cong \wedge^2 (U_7^\perp / U_2)^\vee$ we obtain $\sigma''([U_2\subset U_7])=(\sigma_l|_{U_7^\perp} \mod e_7\wedge e_8)$ of rank at most four; this implies that $[U_2\subset U_7]\in D_{O_2}(\sigma'')$. 

Conversely let us suppose that $[U_2\subset U_7]\in D_{O_2}(\sigma'')^0$. Let us write $\sigma_{v_0}|_{U_7^\perp}$ as $\sigma_{v_0}|_{U_7^\perp}=ue_7\wedge e_8 + e_7\wedge x_7 + e_8\wedge x_8 + \omega$, where $u\in \CC$, $x_7,x_8\in \langle e_2,\cdots,e_6 \rangle$ and $\omega\in \wedge^2 \langle e_2,\cdots,e_6 \rangle$. When we mod out by $ce_7-de_8$ we obtain $(\sigma_{v_0}|_{U_7^\perp} \mod ce_7-de_8) = e_7\wedge (x_7 + \frac{c}{d}x_8)+\omega$. The fact that $[U_2\subset U_7]\in D_{O_2}(\sigma'')^0$ implies that for any $c,d\in \CC$, $e_7\wedge (x_7 + \frac{c}{d}x_8)+\omega$ has rank at most four. For instance, if $\omega$ has already rank four, then $x_7=x_8=0$; then one can check that, by setting $l=\CC (uv_1-v_0)$, the point $[l\subset U_2\subset U_7]\in (X_{1.1}^\sigma)^0$, and as a consequence $[U_2\subset U_7]\in (X_{1.2}^\sigma)^0$. If $\omega$ has rank at most two, one can show analogously that $[U_2\subset U_7]\in (X_{1.2}^\sigma)^0$.
\end{proof}

Let now $k=\CC(\bP^3)$ be the field of rational functions of $\bP(V_{10}/V_6)\cong \bP^3$. The surjective projection $X_{1.2}^\sigma\to \bP(V_{10}/V_6)$ defines a three dimensional $k$-variety $X_{1.2 /k}^\sigma:=X_{1.2}^\sigma \times_{\bP^3}\Spec(k)$, which is a subvariety of $\bP(\cV_7/V_1)_{/k}:=\bP(\cV_7/V_1)\times_{\bP^3}\Spec(k)$. Checking the rationality of $X_{1.2}^\sigma$ then translates into checking the rationality of $X_{1.2 /k}^\sigma$. The following diagram allows to visualize the situation:

\[\begin{tikzcd}
 \bP_k^5 \ar[rd] & X_{1.2/k}\ar[r]\ar[d]\ar[l,hook'] & X_{1.2}^\sigma \ar[d]\\
& \Spec(k)\ar[r,hook] & \bP^3.
\end{tikzcd}\]

Notice that, through the base change over $k$, the bundle $\cN$ becomes a two dimensional $k$-vector space $k^2$, while $\cV_7/V_1$ becomes isomorphic to $k^6$. Then $\bP(\cV_7/V_1)_{/k}\cong \bP^5_k$ is just the projective space over $k$ and by Proposition \ref{prop_description_1.2_quadrics} $X_{1.2 /k}^\sigma$ is just the intersection of two quadrics in this projective space. Indeed, the relative version over $k$ of Lemma \ref{lem_gg_EB} tells us that $\sigma$ defines a section $$\sigma''_k \in H^0(\bP_k^5,\cO_{\bP^5_k}(1)\otimes (\wedge^2 (Q_{\bP^5_k}\oplus k^2)^\vee/\wedge^2 k^2),$$
where $Q_{\bP^5_k}$ is the rank five quotient bundle over ${\bP^5_k}$. The pencil of quadrics defining $X_{1.2 /k}^\sigma$ is given by the Pfaffians of $\sigma_k|_{\cW_6}$ (which is a two-form in six variables), for all bundles $\cW_6$ such that $Q_{\bP^5_k} \subset \cW_6 \cong Q_{\bP^5_k}\oplus k\subset Q_{\bP^5_k}\oplus k^2$. This pencil is therefore parametrized by $\bP(k^2)\cong\bP_k^1$. 

\begin{proposition}
\label{prop_X_k}
$X_{1.2 /k}^\sigma \subset \bP^5_k$ is a three-dimensional intersection of two quadrics. A general quadric in the pencil is smooth but the intersection is singular at a zero dimensional locus (i.e. at a finite number of points, this number being potentially zero).
\end{proposition}

\begin{proof}
The statement about the intersection of quadrics is a consequence of the description of $X_{1.2/k}^\sigma \subset \bP^5_k$ given before. Since $X_{1.2/k}^\sigma=X_{1.2}^\sigma \times_{\bP^3}\Spec(k)$, the statements about the dimension over $k$ of $X_{1.2/k}^\sigma$ and of its singular locus are a consequence of the fact that $X_{1.2}^\sigma$ is six dimensional and its singular locus is three dimensional by Proposition \ref{prop_description_1.2_quadrics} (notice that if the singular locus does not dominate $\bP^3$ then $X_{1.2/k}^\sigma$ is smooth). The fact that a general quadric in the pencil is smooth is again a consequence of Proposition \ref{prop_description_1.2_quadrics}, where it is proved that the fiber of $X_{1.2}^\sigma \to \bP^3$ over a general point in $\bP^3$ is the intersection of a pencil of quadrics whose general element is smooth.
\end{proof}

\subsubsection{Digression: rationality of mildly singular intersection of quadrics}

The rationality of the intersection of two quadrics has been studied in \cite{hassett_tschinkel_rationality_intersection_quadrics}. In this article, the authors work with \emph{smooth} intersection of quadrics, however in the following we will show how their result is directly applicable to our situation. Let us denote by $H_k$ the hyperplane divisor inside $\bP^5_k$.

\begin{lemma}
\label{lem_rat_k}
Let $X_k$ be a $k$-variety which is the intersection of two quadrics inside $\bP_k^n$, $n\geq 4$. Suppose that the general quadric in the pencil is smooth and that the singular locus is zero dimensional, hence it consists of a finite number of points. If there exists a one dimensional geometrically integral subvariety $Z_k\subset X_k$ such that $Z_k\cdot H_k \equiv 1 \pmod{2}$ then $X_k$ is $k$-rational.
\end{lemma}

\begin{proof}
Since the general quadric is smooth, the proof of \cite[Theorem A.5, (i)]{hassett_tschinkel_rationality_intersection_quadrics} applies word for word. Thus $X_k$ contains a linear subspace $\bP_k^1$. By \cite[Proposition 2.2]{ctssd_intersection_quadrics_chatelet_surfaces} $X_k$ is birational to $\bP^{n-2}_k$ or else the $\bP^1_k$ is contained in the singular locus of $X_k$. However, the singular locus of $X_k$ is a finite number of points, so the second case is not possible and $X_k$ is birational to $\bP^3_k$.
\end{proof}

Let us denote by $H$ (respectively $h$) the first Chern class of $\cU_1^\vee|_{X_{1.2}^\sigma}$ (resp. $\cO(1)|_{X_{1.2}^\sigma}$, where $\cO(1)$ is the pullback to $\bP(\cV_7/V_1)$ of the dual of the tautological bundle over $\bP(V_{10}/V_6)\cong\bP^3$).

\begin{theorem}
\label{thm_rat_X1}
Let us suppose that $\sigma\in \cD^{1,6,10}$. If there exists a $4$-dimensional subvariety $Z\subset X_{1.2}^\sigma$ whose intersection with $H\cdot h^3$ is odd and $Z_k:=Z\times_{\bP^3}\Spec(k)$ is geometrically integral, then $X_{1.2/k}^\sigma$ is $k$-rational and $X_{1.2}^\sigma$, $X_{1.1}^\sigma$ and the Peskine sixfold $X_{1}^\sigma$ are rational.
\end{theorem}

\begin{proof}
Notice that the intersection is well defined since we can suppose that $H\cdot h^3$ avoids the singularities of $X_{1.2}^\sigma$. Moreover $Z$ dominates $\bP(V_{10}/V_6)$ since the intersection with $H\cdot h^3$ is non zero. Then combine Lemma \ref{lem_rat_k} and Lemma \ref{prop_X_k} with the fact that $Z$ defines a subvariety $Z_k$ inside $X_{1.2 /k}^\sigma$ which fulfills the hypothesis of Lemma \ref{lem_rat_k}. This implies that $X_{1.2/k}^\sigma$ is rational over $k=\CC(\bP^3)$ and thus that $X_{1.2}^\sigma$ is rational over $\CC$. By Proposition \ref{prop_birational_1.1} and Proposition \ref{prop_birationality_1.2} we deduce the statement of the theorem.
\end{proof}

\begin{remark}
The hypothesis on $Z_k$ being geometrically integral is for instance ensured as soon as $Z$ and all the fibers of $Z\to \bP^3$ are integral (by \cite[Lemma 3.11]{birat_levico}). 
\end{remark}

\section{Open questions}
\label{sec_open_questions}

As in the case of cubic fourfolds containing a plane, an interesting question is whether the results in Theorem \ref{thm_rat_X1} and Proposition \ref{prop_trivial_twisting} are somehow related. In the case of cubics containing a plane, the existence of a codimension two class with the right intersection number implies both that the cubic is rational and that the twisting of the associated twisted K3 surface is trivial. Does a similar result hold for the Peskine sixfold? More generally, we state the following conjecture, which reproduces what happens conjecturally for cubic fourfolds containing a plane:

\begin{Conjecture}
Let $\sigma\in\cD^{1,6,10}$ be a general element.  If the twisting class $\beta$ of the associated degree $6$ twisted K3 surface $(S,\beta)$ is trivial then the Peskine sixfold $X_1^\sigma$ is rational.
\end{Conjecture}

The first step would be to prove that, if there exists a $2$-dimensional subvariety inside $D$ whose intersection with a general fiber of $D\to S$ is odd, so that by Proposition \ref{prop_trivial_twisting} the class $\beta$ is trivial, then the hypothesis of Theorem \ref{thm_rat_X1} are fulfilled. One could speculate even more, asking whether the rationality of a Peskine $X_1^\sigma$ for $\sigma\in \cM$ is equivalent to the existence of an \emph{associated $K3$ surface} (associated $K3$ surfaces may be defined analogously to the case of cubics). This would reproduce the conjectural behaviour of cubic fourfolds in the case of Peskine-Debarre-Voisin varieties and it seems to us that such a question deserves to be investigated further. Notice that Theorem \ref{thm_rationality_3310} would also agree with such a conjecture.

Notice that at this point it is not even clear whether the hypothesis in Theorem \ref{thm_rat_X1} defines countably many divisorial conditions in $\cD^{1,6,10}$. The difficulty in dealing with the Peskine variety $X_1^\sigma$ for $\sigma\in\cD^{1,6,10}$ is the fact that the Peskine (and $X_{1.2}^\sigma$ as well) is not smooth. Explicitly constructing a desingularization $\tilde{X_1^\sigma}\to X_1^\sigma$ should be necessary in order to understand whether the existence of a four dimensional subvariety $Z\subset X_{1.2}^\sigma$ translates into a divisorial condition in the moduli space of $\tilde{X_1^\sigma}$. 

Finally let us mention that some of the constructions in the present paper may be used to understand better the derived category of the Peskine sixfold. In particular, it would be interesting to produce a Kuznetsov component of $K3$-type (inside $D^b(X_1^\sigma)$) which becomes isomorphic to the derived category of a $K3$ surface when an \emph{associated $K3$ surface} exists. We suggest that, when $\sigma\in \cD^{1,6,10}$, the embedding of the cubic fourfold $\cC$ inside $X_1^\sigma$ may be used in order to elucidate better the situation since it is known that the derived category of the (resolution of the) cubic fourfold contains a component of K3-type (by \cite{kuz}).



\end{document}